\documentclass[12pt,a4paper] {article}
\usepackage{latexsym,amsmath,enumerate,amssymb,amsbsy,amsthm}
\usepackage{enumerate,verbatim,tikz,float}

\theoremstyle{plain}

\newtheorem{theorem}{Theorem}
\newtheorem{proposition}{Proposition}
\newtheorem{corollary}{Corollary}
\newtheorem{lemma}{Lemma}

\theoremstyle{remark}

\numberwithin{equation}{section}

\renewcommand\footnotemark{}

\date{}

\begin{document}

	\title{Revisiting the  Factorization of $x^n+1$ over Finite  Fields}

	\author{Arunwan Boripan and Somphong~Jitman}

	\thanks{This research was supported by the Thailand Research Fund under Research
		Grant RSA6280042.}
	
	\thanks{A. Boripan is with the Department of Mathematics, Faculty of Science, Ramkhamhaeng University, Bangkok 10240, Thailand (email: boripan-arunwan@hotmail.com)}

	\thanks{S. Jitman (Corresponding Author)  is with the  Department of Mathematics, Faculty of Science,
		Silpakorn University, Nakhon Pathom 73000,  Thailand
		(email: sjitman@gmail.com).}

	\maketitle

\begin{abstract}
The polynomial $x^n+1$ over finite fields has been of interest due to its applications in the study of negacyclic codes over finite fields. In this paper, a rigorous treatment of the factorization of $x^n+1$ over finite fields is given as well as   its applications.  Explicit and recursive  methods  for   factorizing $x^n+1$ over finite fields are provided together with the enumeration formula.   
As applications, some families of  negacyclic codes are revisited with more clear and   simpler  forms. 

\end{abstract}

    \noindent{\bf keywords}: {Factorization, Enumeration, Polynomials,   Negacyclic Codes}


	\section{Introduction} 
 In coding theory,  the polynomial $x^n+1$ over finite fields       plays an important role in the study of negacyclic codes (see \cite{BR2013}, \cite{B2008}, \cite{JPR2019}, \cite{SJLU2015},  and references therein).   Precisely, a negacyclic code of  length $n$ over  $\mathbb{F}_q$  can be    uniquely  determined by  an    ideal  in the principal ring  $\mathbb{F}_q[x]/\langle x^n+1\rangle $  generated by a monic divisor of $x^n+1$. 
A brief discussion on the factorization of $x^n+1$ over finite fields $\mathbb{F}_q$   has been given in  \cite{JPR2019} and \cite{SJLU2015}.  In the case where the characteristic of  $\mathbb{F}_q$  is even, the  factorization of  $x^n+1=x^n-1$ over  $\mathbb{F}_q$ has been given  and applied in the study of cyclic codes over finite fields  in \cite{JLX2011}. 
In  \cite{BGM1993} and \cite{M1996}, an explicit form of the  factorization of $x^{2^i}+1$ over finite fields of odd characteristic has been established.

In this paper, we focus on the factorization of $x^n+1$ over finite fields $\mathbb{F}_q$  for  arbitrary positive integers $n$ and all odd prime powers $q$.  If the characteristic of $\mathbb{F}_q$ is $p$, we have  \[x^{p^sn}+1= (x^{n}+1)^{p^s}\] for all integers $n\geq 1$  and  $s\geq 0$. It is therefore  sufficient  to study  the factorization of $x^n+1$ over $\mathbb{F}_q$ such that  $n$ is  co-prime to $q$.
Here, we write $n=2^in'$  for some  integer $i\geq 0$ and odd positive integer $n'$ such that $\gcd(n',q)=1$.   

Before proceed to  the general results,  we  consider  a pattern on  the    factorization of $x^{2^i11}+1 $ over $\mathbb{F}_5$. We have 
\begin{align*} 
x^{2\cdot 11}+1
=&f_1(x)f_2(x)f_3(x)f_4(x)f_5(x)f_6(x)\\
x^{2^2\cdot 11}+1=&f_1(x^2)f_2(x^2)f_3(x^2)f_4(x^2)f_5(x^2)f_6(x^2)\\ 
\vdots&\\
x^{2^i\cdot 11}+1=&f_1(x^{2^{i-1}})f_2(x^{2^{i-1}})f_3(x^{2^{i-1}})f_4(x^{2^{i-1}})f_5(x^{2^{i-1}})f_6(x^{2^{i-1}})
\end{align*}
for all $i\geq 1$, where $f_1(x)=x + 2$, $f_2(x)=x + 3$, $f_3(x)=x^5 + x^4 + x^3 + 2x^2 + x + 2$, $f_4(x)=x^5 + 2x^4 + x^3 + 2x^2 + 3x + 2$, $f_5(x)=x^5 + 3x^4 + x^3 + 3x^2 + 3x + 3$ and $f_6(x)=x^5 + 4x^4 + x^3 + 3x^2 + x + 3$.   It is easily seen that the   factorization  can be determined recursively on the exponent $i$ of $2$ and the number of  monic irreducible factors of  $x^{2^i11}+1 $ is  a constant independent of $i\geq 2$.  

In this paper,  a complete study on  the above pattern of the factorization of $x^{2^in'}+1 $  over $\mathbb{F}_q$ is given.  Precisely, we prove that there exists a positive integer $k$ such that  the number of monic irreducible factors of $x^{2^in'}+1$  over $\mathbb{F}_q$ becomes a constant for all positive integers $i\geq k$.
In the cases where ${\rm ord}_{n'}(q)$ is odd,   a complete recursive factorization of $x^{2^in'}+1$  over $\mathbb{F}_q$  is provided  together with   a recursive formula for the number of its monic irreducible factors for all positive integers $i$.   In the cases where  ${\rm ord}_{n'}(q)$ is even, a   recursive factorization of  $x^{2^in'}+1$  over $\mathbb{F}_q$  is given  for all positive integers $i\geq k$.   As applications, constructions and enumerations of some negacyclic codes  of lengths $2^in'$ over  $\mathbb{F}_q$ are given  based on the above results.

The paper is organized as follows.  Preliminary concepts and results on the factorization of  $x^n+1$ over finite fields are recalled in Section~\ref{sec2}. In Section~\ref{sec3},   the number theoretical results and properties of $q$-cyclotomic cosets required in the study of   the factorization of  $x^{2^in'}+1$ are established.  Recursive methods for factorizing $x^{2^in'}+1$  and   enumerating  its monic irreducible factors are given in Section~\ref{sec4}. Applications   in the study of negacyclic codes over finite fields are revisited  in Section~\ref{sec5}.  

\section{Preliminary}

\label{sec2}

In this section, basic concepts  and tools  used in the study of  the  factorization of  $x^n+1$  over finite fields  and the enumeration of its monic irreducible factors  are recalled.   

For a positive integer $a$  and an  integer $s$,  the notation  $2^s||a $ is used whenever  $s$ is the largest integer such that $a$ is divisible by $2^s$, or equivalently, $2^s|a$ but $2^{s+1}\nmid a$.  For  an integer $a$ and a positive integer $n$, denote by  ${\Theta_{n}(a)}$ the additive order of $a$ modulo $n$.  In the case where $\gcd(a,n)=1$, denote by  ${{\rm  ord}_{ n}(a)}$ the multiplicative order of $a$ modulo $n$.  By abuse of notation,  we write  ${{\rm  ord}_{ 1}(a)}=1$.

For a prime power  $q$,  a positive integer  $n$     co-prime to $q$,  and an integer $0\leq a<n$,  the {\em $q$-cyclotomic  coset  modulo  $n$  containing $a$} is defined to be  
\begin{align*}
Cl_{q,n}(a) =\{aq^{j} \, (\mathrm{mod}~ n) \mid j =0,1,2,\dots\}.
\end{align*}
It is not difficult to see that  
$
Cl_{q,n}(a)  =\{aq^{j} \, (\mathrm{mod } \,n) \mid  0\leq j < \mathrm{ord}_{{\Theta}_n(a)}{(q)}\}$
and $|Cl_{q,n}(a)  |=\mathrm{ord}_{{\Theta}_n(a)}{(q)}$. Moreover, $ {\Theta}_n(a)= {\Theta}_n(j)$ for all  $j\in Cl_{q,n}(a)$.
Let $S_{q}(n)$ denote a complete set of representatives of the $q$-cyclotomic  cosets modulo  $n$ and let $\alpha$ be a primitive $n$th root of unity in  some  extension field of $\mathbb{F}_{q}$. It is well known (see \cite{LSBook}) that 
\begin{align}\label{xn-1}
x^n-1=\prod\limits_{a\in S_{q}(n)} f_a(x),
\end{align}
where
\begin{align}\label{fa}  f_a(x)=\prod\limits_{j\in Cl_{q,n}(a) }{(x-\alpha^j)}\end{align}
is  the minimal polynomial   of $\alpha^a$  over $\mathbb{F}_{q}$ referred as  the irreducible polynomial   induced by $Cl_{q,n}(a)$.

 In \cite{BJ2019},  a basic idea  for the factorization of $x^{2^in'}+1$ is  given  using $\eqref{xn-1} $ and the following lemmas. 

\begin{lemma}[{\cite[Lemma 2]{BJ2019}}] \label{srim-parity} Let $q$ be an odd prime power and let $n'$ be an odd positive integer such that $\gcd(q,n')=1$. Let  $i\geq 0$ and   $0\leq a < 2^{i+1}n'$ be integers. Then   the  elements in $Cl_{q, 2^{i+1}n'} (a)$  have the same parity.
\end{lemma}

\begin{lemma}[{\cite[Lemma 3]{BJ2019}}]  \label{div-1}  Let $q$ be an odd prime power and let $n'$ be an odd positive integer such that $\gcd(q,n')=1$. Let  $i\geq 0$ and   $0\leq a < 2^{i+1}n'$ be integers.  Then the polynomial  $f_{a}(x)$ induced by $Cl_{q, 2^{i+1}n'} (a)$  is a divisor of $ x^{2^in'}+1$ if and only if 
	  $a$ is odd.
\end{lemma}

From Lemma \ref{srim-parity}, the parity of   a representative of $Cl_{q, 2^{i+1}n'} (a)$ is independent of its choices.  By Lemma \ref{div-1}, the monic  irreducible divisors of  $ x^{2^in'}+1$ are induced by  the   $q$-cyclotomic  cosets  modulo $2^{i+1}n'$ containing odd integers.  Let $SO_{q}(n)$ (resp., $SE_{q}(n)$) denote a complete set of representatives of the $q$-cyclotomic  cosets  containing odd integers  (resp., even integers) modulo  $n$.  It follows that  \begin{align} \label{facn0}
x^{2^i n^\prime}+1&=\dfrac{x^{2^{i+1} n^\prime}-1}{x^{2^i n^\prime}-1}  = 
 \prod\limits_{a\in SO_{q}(2^{i+1} n')} f_a(x)
\end{align}
for all $i\geq 0$.

For a positive integer $n$ and a prime power $q$, let $N_q(n)$ denote the number of monic irreducible  factors of $x^n+1$ over $\mathbb{F}_q$.  Based on \cite[Equation (3.1)]{JPR2019}, it can be deduced that

\begin{align} \label{nqn'}
N_q(2^i n')=\sum_{d\mid n'} \frac{\phi(2^{i+1}d)}{{\rm  ord}_{2^{i+1}d}(q)}.
\end{align}

As discussed above, the $q$-cyclotomic cosets modulo $2^{i+1}n'$ containing odd integers   are key to determine the factorization of  $x^{2^in'}+1$ over $\mathbb{F}_q$ and the enumeration of its monic irreducible factors.   Properties of these cosets  are study in the next section.

%
%
%

\section{Number Theoretical Results and  Cyclotomic Cosets}  \label{sec3}

In this section, number theoretical  results required in the factorization of $x^{2^in'}+1$ are derived. Subsequently,   properties of $q$-cyclotomic cosets modulo $2^{i+1}n'$ containing odd integers  are established for all positive integers $i$ and odd positive integers $n'$. These results are key   in the study of  the factorization of $x^{2^in'}+1$  in Section \ref{sec4}.

 A relation on the carnality of the $q$-cyclotomic costs containing  odd integers $a$ and $a+2^in'$ modulo $2^{i+1}n'$ is given in the following lemma.
 
 \begin{lemma} \label{lemCardinal} Let $q$ be an odd prime power   and let $n'$ be an odd positive integer such that $\gcd(q,n')=1$.    Then $|Cl_{q,2^{i+1}n'}(a)|=|Cl_{q,2^{i+1}n'}(a+2^in')|$   all odd    integers $a$ and  for all positive integers   $i$.
 	
 \end{lemma}
 \begin{proof} Let $a$ be an odd integer and let $i$ be a positive integer.  
 	Then  \begin{align*}
 	{\Theta_{2^{i+1}n'}(a)}&= \dfrac{{2^{i+1}n' }}{\gcd(2^{i+1}n',a) } =\dfrac{{2^{i+1}n'}}{\gcd(n',a)} =\dfrac{2^{i+1}n'}{\gcd(n',a+2^in')}\\
 	& = \dfrac{2^{i+1}n'}{\gcd(2^{i+1}n',a+2^in') } = {\Theta_{2^{i+1}n'}(a+2^in')}. \end{align*}
 	Hence,  	\[|Cl_{q,2^{i+1}n'}(a)|= {\rm ord}_{\Theta_{2^{i+1}n'}(a)}(q) = {\rm ord}_{\Theta_{2^{i+1}n'}(a+2^in')}(q) =|Cl_{q,2^{i+1}n'}(a+2^in')|\]
 	as desired.
 \end{proof}
 
 Properties  of $q$-cyclotomic cosets   with  $q \equiv 1 \ ({\rm mod }\ 4) $  and $q \equiv 3 \ ({\rm mod }\ 4) $ are given separately  in the following subsections. 
 
 \subsection{$q \equiv 3 \ ({\rm mod }\ 4) $}
   In this subsection, we focus on properties of $q$-cyclotomic  cosets in the case where $q \equiv 3 \ ({\rm mod }\ 4) $.

First, we determine an explicit formula  for  ${\rm ord}_{2^i}(q)$    for all odd prime powers $q \equiv 3 \ ({\rm mod }\ 4) $ and  positive integers $i$.

\begin{lemma} \label{ord2iq3}Let  $q$ be an odd prime power and let $\beta$ be   the positive integer such that $2^\beta|| (q^2-1)$.   Let $i$ be a positive integer.    If $q \equiv 3 \ ({\rm mod }\ 4) $, then 
		\[ {\rm ord}_{2^i}(q)= \begin{cases} 1 & \text{ if } i=1,\\
		2& \text{ if }2\leq i\leq \beta,\\ 
		2^{i-\beta+1} & \text{ if }i\geq \beta+1 .
		\end{cases}\]
\end{lemma}
\begin{proof}
Assume that $q \equiv 3 \ ({\rm mod }\ 4) $.  Then $2||(q-1)$  and  $2^{i}|(q^2-1)$ for all $2\leq i\leq \beta$.   Since $q^3-1=(q-1)(q^2+q+1)$ and $q^2+q+1$ is odd, we have  $2||(q^3-1)$. Hence, ${\rm ord}_{2}(q)= 1$   and ${\rm ord}_{2^i}(q)= 2$  for all $2\leq i\leq \beta$.   

Assume that   $i\geq \beta+1 $.  Since $q \equiv 3 \ ({\rm mod }\ 4) $,  it follows that $q^{2^j} \equiv 1 \ ({\rm mod }\ 4) $ for all $j\geq 1$.  Hence, $2||(q^{2^{j}} +1) $  for all $j\geq 1$. 
Since   
$(q^{2^{i-\beta}} -1)(q^{2^{i-\beta}} +1)=
q^{2^{i-\beta+1}} -1= (q^2 -1) \prod\limits_{j=1}^{i-\beta}(q^{2^{j}} +1) $, we have  ${2^{i}}||(q^{2^{\beta-i+1}} -1)$  and  $2^{i} \nmid (q^{2^{t}} -1)$  for all $t\leq \beta+i$.  Hence,  
	${\rm ord}_{2^{i}}(q)= 2^{i-\beta+1}$  for all $i\geq \beta+1$. 
\end{proof}

 	Properties  of  $q$-cyclotomic cosets modulo $2^{i+1}n'$ containing odd integers  are established in the next  proposition.  
\begin{proposition} \label{propQ3} Let $q$ be a prime power such that $q \equiv 3 \ ({\rm mod }\ 4) $ and let $n'$ be an odd positive integer such that $\gcd(q,n')=1$.    Let  $\lambda\geq 0$ be the integer such that $2^\lambda||  {\rm ord}_{n'}(q)$ and let  $\beta$ be   the positive integer such that $2^\beta|| (q^2-1)$.     Then  the following statements hold.
	\begin{enumerate}[$i)$]

		\item  If $\lambda=0$, then  the following statements hold.
		\begin{enumerate}[$a)$]
			
			\item  $Cl_{q,2^{i+1}n'}(a)\neq Cl_{q,2^{i+1}n'}(a+2^in')$ all odd    integers $a$ and    integers  $2\leq i\leq \beta-1.$
			\item  $Cl_{q, 2^{i+1}n'} (a)   =Cl_{q, 2^{i+1}n'} (a+2^in')   = Cl_{q, 2^{i}n'} (a) \cup(Cl_{q, 2^{i}n'} (a) +2^{i}n')$    for all odd  integers $a$ and   integers    $  i =1$ or $i\geq \beta$.
		\end{enumerate}
		
		\item  If $\lambda>0$, then  the following statements hold. 
		
		\begin{enumerate}[$a)$]
			\item   $ Cl_{q, 2^{ \lambda+\beta-1}n'} (1) \ne  Cl_{q, 2^{\lambda+\beta-1}n'} (1+2^{\lambda+\beta-2}n')$.
			\item  $Cl_{q, 2^{i+1}n'} (a)   = Cl_{q, 2^{i}n'} (a) \cup(Cl_{q, 2^{i}n'} (a) +2^{i}n')$    for all odd   integers $a$ and    integers $  i \geq \lambda+\beta-1$. 
		\end{enumerate}
	\end{enumerate}
\end{proposition}
\begin{proof} First, we observe that $\beta\geq 3$, $2||(q-1)$ and $2^{\beta-1}||(q+1)$.

	To prove $i)$, assume that  $\lambda=0$.  In this case,  $ {\rm ord}_{n'}(q)$ is odd which implies that ${\rm ord}_{{\Theta_{n'}(a)}}(q) $ is odd   for all odd positive integers $a$.

	To prove $a)$,  let $a$ be an odd integer and let $i$ be an integer such that $2\leq i\leq \beta-1$.   By Lemma \ref{ord2iq3},  it follows that   ${\rm ord}_{2^{i}}(q) =2={\rm ord}_{2^{i+1}}(q) $. 
	Since ${\rm ord}_{{\Theta_{n'}(a)}}(q) $ is odd,  it can be deduced that  $ {\rm ord}_{{\Theta_{2^{i+1}n'}(a)}}(q)=   {\rm ord}_{2^{i+1}{\Theta_{n'}(a)}}(q) ={\rm lcm} ( {\rm ord}_{2^{i+1}}(q) ,  {\rm ord}_{{\Theta_{n'}(a)}}(q) ) 
	={\rm lcm} ( {\rm ord}_{2^{i}}(q) ,  {\rm ord}_{{\Theta_{n'}(a)}}(q) ) =  {\rm ord}_{{\Theta_{2^{i}n'}(a)}}(q)$.  
	Suppose that   $Cl_{q,2^{i+1}n'}(a)= Cl_{q,2^{i+1}n'}(a+2^in').$  Since $a \not\equiv a+2^in' \, (\textrm{mod} \,2^{i+1}n') $, there exists $0< j< {\rm ord}_{{\Theta_{2^{i+1}n'}(a)}}(q)$ such that  $a+2^in'\equiv \,aq^j \, (\textrm{mod} \,2^{i+1}n').$ Hence,   we have  $a\equiv \,aq^j \, (\textrm{mod} \,2^{i}n') $ which implies that $ {\rm ord}_{{\Theta_{2^{i}n'}(a)}}(q)\leq j <   {\rm ord}_{{\Theta_{2^{i+1}n'}(a)}}(q) = {\rm ord}_{{\Theta_{2^{i}n'}(a)}}(q)$, a contradiction.  	Therefore, $Cl_{q,2^{i+1}n'}(a)\neq Cl_{q,2^{i+1}n'}(a+2^in')$ as desired.

		To prove $b)$,   
		let $a$ be an odd integer and let $i$ be an integer such that  $i =1$ or $i\geq \beta$.       By Lemma \ref{ord2iq3},  we have ${\rm ord}_{2^{i+1}}(q)= 2 {\rm ord}_{2^i}(q)$.  Since ${\rm ord}_ {n'}(q) $ is odd, we have $ {\rm ord}_{{{2^{i+1}n'}}}(q)  ={\rm lcm} ( {\rm ord}_{2^{i+1}}(q) ,  {\rm ord}_{ n'}(q) ) 
		={\rm lcm} ( 2{\rm ord}_{2^{i}}(q) ,  {\rm ord}_{ n'}(q) ) =2  {\rm ord}_{{{2^{i}n'}}}(q)$ which implies that  $a q^{ {\rm ord}_{{{2^{i}n'}}}(q)} \not\equiv   a  \, ({\rm mod}\, 2^{i+1}n')$. 
		Since 	$a q^{ {\rm ord}_{{{2^{i}n'}}}(q)} \equiv   a  \, ({\rm mod}\, 2^{i}n')$, we have  	  $a q^{ {\rm ord}_{{{2^{i}n'}}}(q)} \equiv   a +2^in'  \, ({\rm mod}\, 2^{i+1}n')$.    Hence,  $a+2^in'\in Cl_{q, 2^{i+1}n'} (a)$ which implies that   $Cl_{q, 2^{i+1}n'} (a)   =Cl_{q, 2^{i+1}n'} (a+2^in')$.  This proves the first equality.

		  For the second equality, let $b\in  Cl_{q, 2^{i+1}n'} (a)$.  Then   $b\equiv aq^j \, ({\rm mod }\, 2^{i+1}n')$  for some $0\leq j< {\rm ord}_{{\Theta_{2^{i+1}n'}(a)}}(q)$.   It follows that   $b\equiv aq^j \, ({\rm mod }\, 2^{i}n')$.   If $b<2^in'$, then $b\in Cl_{q, 2^in'} (a)$. Otherwise, $b-2^in'\in   Cl_{q, 2^in'} (a)$ which implies that   $b  \in Cl_{q, 2^in'} (a)  +2^in'$. Hence,  $Cl_{q, 2^{i+1}n'} (a)   \subseteq  Cl_{q, 2^in'} (a) \cup( Cl_{q, 2^in'} (a) +2^in')$. 
		Since $ Cl_{q, 2^in'} (a) $  and $ Cl_{q, 2^in'} (a) +2n'$ are disjoint sets of the same size $ {\rm ord}_{{\Theta}_{2^in'}(a)}(q)$, we have  $|Cl_{q, 2^{i+1}n'} (a) | = {\rm ord}_{{\Theta_{2^{i+1}n'}(a)}}(q)= 2  {\rm ord}_{{\Theta_{2^{i}n'}(a)}}(q)  = |Cl_{q, 2^in'} (a) \cup( Cl_{q, 2^in'} (a) +2^in')|$. 
Therefore, $Cl_{q, 2^{i+1}n'} (a)   = Cl_{q, 2^in'} (a) \cup( Cl_{q, 2^in'} (a) +2^in')$  as desired.

	To prove $ii)$, assume that  $\lambda>0$.    For  $a)$,  suppose that $1\in Cl_{q, 2^{\lambda+ \beta-1}n'} (1+2^ {\lambda+ \beta-2}n')$.  If $\lambda=1$,  then  $\lambda+\beta-1 = \beta$, 
		we have   ${\rm ord}_{2^{\lambda+\beta-1}}(q) =2={\rm ord}_{2^{\lambda+\beta-2}}(q) $  by Lemma \ref{ord2iq3}. 	Since $2||{\rm ord}_{ n'}(q) $, we have $\frac{{\rm ord}_{ n'}(q)}{2} $ is odd  and it  follows that $ {\rm ord}_{{ {2^{  \lambda+\beta-1 }n'}}}(q)  ={\rm lcm} ( {\rm ord}_{2^{\lambda+\beta-1 }}(q) ,  {\rm ord}_{{{n'}}}(q) )
		={\rm lcm} ( {\rm ord}_{2^{\lambda+\beta-2}}(q) ,  {\rm ord}_{{\Theta_{n'}(a)}}(q) ) =  {\rm ord}_{{\Theta_{2^{\lambda+\beta-2}n'}(a)}}(q)$.  
			Assume  that $\lambda\geq 2$. 
		Since $\lambda+\beta-1 \geq \beta+1$, 
		 we have   ${\rm ord}_{2^{\lambda+\beta-1}}(q) = 2^\lambda $  and  ${\rm ord}_{2^{\lambda+\beta-2}}(q)=2^{ \lambda -1}  $ by Lemma \ref{ord2iq3}.
		 Since $2^\lambda|| {\rm ord}_ {n'}(q) $, it follows that   \begin{align*}
		 {\rm ord}_{{{2^{\lambda+\beta-1}n'}}}(q)  &={\rm lcm} ( {\rm ord}_{2^{ \lambda+\beta-1}}(q) ,  {\rm ord}_{ n'}(q) ) \\&={\rm lcm} (  2^\lambda,  {\rm ord}_{ n'}(q) ) 
		\\& ={\rm lcm} (  2^{\lambda-1},  {\rm ord}_{ n'}(q) ) 
		\\& ={\rm lcm} ( {\rm ord}_{2^{ \lambda+\beta-2}}(q) ,  {\rm ord}_{ n'}(q) ) 
		\\& = {\rm ord}_{{{2^{\lambda+\beta-2}n'}}}(q). 
		 \end{align*} 
		 	 Since  \  $2^\lambda||{\rm ord}_{ n'}(q) $,     $\frac{{\rm ord}_{ n'}(q)}{2^\lambda} $ is odd.  Hence,   $ {\rm ord}_{{ {2^{\lambda+\beta-1}n'}}}(q)  ={\rm lcm} ( {\rm ord}_{2^{\lambda+\beta-1}}(q) ,  {\rm ord}_{ n'}(q) )  $ $
		={\rm lcm} ( {\rm ord}_{2^{\lambda+\beta-2}}(q) ,  {\rm ord}_{ n'}(q) ) =  {\rm ord}_{{ {2^{\lambda+\beta-2}n'}}}(q)$.   	Since  $1+2^{\lambda+\beta-2} n'\not  \equiv 1\,  ({\rm mod}\, 2^{{\lambda+\beta-1}}n')$,  we have  $1+2^{\lambda+\beta-2}n' \equiv q^j\,  ({\rm mod}\, 2^{{\lambda+\beta-1}}n')$ for some $0<j<   {\rm ord}_{{\Theta_{2^{{\lambda+\beta-1}}n'}(1)}}(q)=  {\rm ord}_{{{2^{{\lambda+\beta-1}}n'}}}(q)$. It follows that $1 \equiv q^j\,  ({\rm mod}\, 2^{{\lambda+\beta-2}}n')$ which implies that  ${\rm ord}_{{{2^{{\lambda+\beta-2}}n'}}}(q) \leq  j<     {\rm ord}_{{{2^{{\lambda+\beta-1}}n'}}}(q)={\rm ord}_{{{2^{{\lambda+\beta-2}}n'}}}(q) $, a contradiction. Therefore, $ Cl_{q, 2^{ \lambda+\beta-1}n'} (1) \ne  Cl_{q, 2^{\lambda+\beta-1}n'} (1+2^{\lambda+\beta-2}n')$ as desired. . 
	
  To prove $b)$,   
	let $a$ be an odd integer and let $i$ be an integer such that $i\geq \lambda+\beta-1$.       Then  $i  \geq \beta$ which implies that ${\rm ord}_{2^{i+1}}(q)= 2 {\rm ord}_{2^i}(q)$  and  ${\rm ord}_{2^i}(q)=2^{i-\beta+1} \geq 2^\lambda $   by Lemma \ref{ord2iq3}. 
 	Since $2^\lambda|| {\rm ord}_ {n'}(q) $,  $ \frac{{\rm ord}_{ n'}(q)}{2^\lambda} $ is odd and  \begin{align*}
 	{\rm ord}_{{{2^{i+1}n'}}}(q)  &={\rm lcm} ( {\rm ord}_{2^{i+1}}(q) ,  {\rm ord}_{ n'}(q) ) 
 	\\ &={\rm lcm} ( 2{\rm ord}_{2^{i}}(q) ,  {\rm ord}_{ n'}(q) ) 
 	\\& ={\rm lcm} ( 2{\rm ord}_{2^{i}}(q) ,  \frac{{\rm ord}_{ n'}(q)}{2^\lambda} ) 
 \\&	=2{\rm lcm} ( {\rm ord}_{2^{i}}(q) ,  \frac{{\rm ord}_{ n'}(q)}{2^\lambda} ) 
 \\&	=2{\rm lcm} ( {\rm ord}_{2^{i}}(q) ,  {\rm ord}_{ n'}(q) ) 
 \\&	=2  {\rm ord}_{{{2^{i}n'}}}(q)
 	\end{align*}  which implies that  $a q^{ {\rm ord}_{{{2^{i}n'}}}(q)} \not\equiv   a  \, ({\rm mod}\, 2^{i+1}n')$. 
	Since 	$a q^{ {\rm ord}_{{{2^{i}n'}}}(q)} \equiv   a  \, ({\rm mod}\, 2^{i}n')$, we have  	  $a q^{ {\rm ord}_{{{2^{i}n'}}}(q)} \equiv   a +2^in'  \, ({\rm mod}\, 2^{i+1}n')$.    Hence, $a+2^in'\in Cl_{q, 2^{i+1}n'} (a)$ which implies that   $Cl_{q, 2^{i+1}n'} (a)   =Cl_{q, 2^{i+1}n'} (a+2n')$.  The first equality  holds.

	For the second equality, let $b\in  Cl_{q, 2^{i+1}n'} (a)$.  Then   $b\equiv aq^j \, ({\rm mod }\, 2^{i+1}n')$  for some $0\leq j< {\rm ord}_{{\Theta_{2^{i+1}n'}(a)}}(q)$.   It follows that   $b\equiv aq^j \, ({\rm mod }\, 2^{i}n')$.   If $b<2^in'$, then $b\in Cl_{q, 2^in'} (a)$. Otherwise, $b-2^in'  \in Cl_{q, 2^in'} (a)$ which implies that   $b\in Cl_{q, 2^in'} (a)  +2^in'$. Hence,  $Cl_{q, 2^{i+1}n'} (a)   \subseteq  Cl_{q, 2^in'} (a) \cup( Cl_{q, 2^in'} (a) +2^in')$. 
	Since $ Cl_{q, 2^in'} (a) $  and $ Cl_{q, 2^in'} (a) +2^in'$ are disjoint sets of the same size $ {\rm ord}_{{\Theta}_{2^in'}(a)}(q)$, we have  $|Cl_{q, 2^{i+1}n'} (a) | = {\rm ord}_{{\Theta_{2^{i+1}n'}(a)}}(q)= 2  {\rm ord}_{{\Theta_{2^{i}n'}(a)}}(q)  = |Cl_{q, 2^in'} (a) \cup( Cl_{q, 2^in'} (a) +2^in')|$. 
	Therefore, $Cl_{q, 2^{i+1}n'} (a)   = Cl_{q, 2^in'} (a) \cup( Cl_{q, 2^in'} (a) +2^in')$  as desired.
\end{proof}

 \subsection{$q \equiv 1 \ ({\rm mod }\ 4) $}
 Here, we investigate  properties of $q$-cyclotomic  cosets in the case where $q \equiv 1 \ ({\rm mod }\ 4) $.  We begin with an explicit formula  for  ${\rm ord}_{2^i}(q)$.

\begin{lemma} \label{ord2iq1}Let  $q$ be an odd prime power and let $\beta$ be   the positive integer such that $2^\beta|| (q^2-1)$.   Let $i$ be a positive integer.   If $q \equiv 1 \ ({\rm mod }\ 4) $, then 
	\[{\rm ord}_{2^{i}}(q)= \begin{cases}
	1 &  \text{ if  } 1\leq i\leq \beta-1,\\ 
	2^{i-\beta+1}& \text{  if  }i\geq \beta.  
	\end{cases} \]
\end{lemma}
\begin{proof}
	Assume that $q \equiv 1 \ ({\rm mod }\ 4) $.  Then $2^{\beta-1}||(q-1)$   which implies that   ${\rm ord}_{2^{i}}(q)= 1$  for all $1\leq i\leq \beta-1$.  Next, assume that   $i\geq \beta $.  Since $q \equiv 1 \ ({\rm mod }\ 4) $,  it follows that  $q^{2^j} \equiv 1 \ ({\rm mod }\ 4) $ for all $j\geq 0$.  Hence, $2||(q^{2^{j}} +1) $  for all $j\geq 0$. 
	Since   
	$(q^{2^{i-\beta}} -1)(q^{2^{i-\beta}} +1)=
	q^{2^{i-\beta+1}} -1= (q -1) \prod\limits_{j=0}^{i-\beta}(q^{2^{j}} +1) $, it can be concluded that   ${2^{i}}||(q^{2^{\beta-i+1}} -1)$  and  $2^{i} \nmid (q^{2^{t}} -1)$  for all $t\leq \beta+i$. As desired, we have 
	${\rm ord}_{2^{i}}(q)= 2^{i-\beta+1}$  for all $i\geq \beta$.
\end{proof}

\begin{proposition} \label{propQ1} Let $q$ be a prime  power such that $q \equiv 1 \ ({\rm mod }\ 4) $ and let $n'$ be an odd positive integer such that $\gcd(q,n')=1$.    Let  $\lambda\geq 0$ be the integer such that $2^\lambda||  {\rm ord}_{n'}(q)$ and let  $\beta$ be   the positive integer such that $2^\beta| |(q^2-1)$.       Then following statements hold.

	\begin{enumerate}[$i)$]
		\item If $\lambda=0$, then 
		\begin{enumerate}[$a)$]
			\item 	  $Cl_{q,2^{i+1}n'}(a)\neq Cl_{q,2^{i+1}n'}(a+2^in')$   for all odd  integers $a$ and   integers  $1\leq i\leq \beta-2.$
			\item   $Cl_{q, 2^{i+1}n'} (a)   = Cl_{q, 2^{i+1}n'} (a+2^in')   = Cl_{q, 2^{i}n'} (a) \cup(Cl_{q, 2^{i}n'} (a) +2^{i}n')$    for all odd integers $a$ and  integers      $  i \geq \beta-1$. 
		\end{enumerate}
		
		\item If $\lambda>0$, then 
		\begin{enumerate}
			\item  $ Cl_{q, 2^{ \lambda+\beta-1}n'} (1) \ne  Cl_{q, 2^{i\lambda+\beta-1}n'} (1+2^{\lambda+\beta-2}n')$. 
			\item   $Cl_{q, 2^{i+1}n'} (a)   = Cl_{q, 2^{i+1}n'} (a+2^in')   = Cl_{q, 2^{i}n'} (a) \cup ( Cl_{q, 2^{i}n'} (a) +2^{i}n')$    for odd  integers $a$ and    integers $  i \geq \lambda+\beta-1$. 
		\end{enumerate}
	\end{enumerate}
\end{proposition}
\begin{proof}

	First, we observe  that $\beta\geq 3$, $2||(q+1)$ and $2^{\beta-1}||(q-1)$.  Using Lemma \ref{ord2iq1}  and arguments similar to those in the proof of  Proposition \ref{propQ3}, it can be deduced  the following key results.   
	\begin{enumerate}[$1)$]
		\item  If $\lambda=0$, then  ${\rm ord}_{{\Theta_{2^{i+1}n'}(a)}}(q) = {\rm ord}_{{\Theta_{2^{i}n'}(a)}}(q)$ for all odd integers $a$ and integers  $1\leq i \leq \beta-2$, and 
		$ {\rm ord}_{{{2^{i+1}n'}}}(q)   =2  {\rm ord}_{{{2^{i}n'}}}(q)$ for  all integers $ i \geq \beta-1$. 
		\item  If $\lambda>0$, then  $ {\rm ord}_{{ {2^{\lambda+\beta-1}n'}}}(q)   =  {\rm ord}_{{ {2^{\lambda+\beta-2}n'}}}(q)$, and 
		$ {\rm ord}_{{{2^{i+1}n'}}}(q)   =2  {\rm ord}_{{{2^{i}n'}}}(q)$  for all integers $i\geq  \lambda+\beta-1$. 
	\end{enumerate}
	The complete proof can be  obtained using the arguments similar to those in Proposition \ref{propQ3}  while the above discussion and Lemma \ref{ord2iq1} is applied instead of Lemma \ref{ord2iq3}.	
\end{proof}

\section{Factorization of $x^n+1$ over Finite Fields}\label{sec4}
In this section,     the factorization of  $x^{2^in'}+1$  over $\mathbb{F}_q$ is established. First, we prove that there exists a positive integer $k$ such that  the number of monic irreducible factors of $x^{2^in'}+1$  over $\mathbb{F}_q$ becomes a constant for all   integers $i\geq k$.
In the case where  ${\rm ord}_{n'}(q)$ is odd,    a complete recursive factorization of $x^{2^in'}+1$  over $\mathbb{F}_q$   is given together with a recursive formula for the number of its monic irreducible factors   for all positive integers $i$ in Subsection  \ref{subsec4.2}.  In the case where  ${\rm ord}_{n'}(q)$ is even,  a  recursive factorization of $x^{2^in'}+1$  over $\mathbb{F}_q$  is given  as well as a recursive formula for the number of  its monic irreducible factors  for all   integers $i\geq k$  in Subsection  \ref{subsec4.1}.

\subsection{Recursive Factorization  of $x^n+1$ over $\mathbb{F}_q$ with  Odd ${\rm ord}_{n'}(q)$} \label{subsec4.2}
In this subsection, we established a complete recursive factorization of  $x^{2^in'}+1$  over $\mathbb{F}_q$  in the case where  ${\rm ord}_{n'}(q)$ is odd.  Subsequently, a  formula for the number of monic irreducible factors of $x^{2^in'}+1$  over $\mathbb{F}_q$ is given  recursively on~$i$.

\subsubsection{$q \equiv 3 \ ({\rm mod }\ 4) $}
 We begin with  useful relations between  $q$-cyclotomic cosets  and their    induced  polynomials   for  the case $q \equiv 3 \ ({\rm mod }\ 4) $. 
\begin{lemma}   \label{lemClPoly3} Let $q$ be a prime  power such that $q \equiv 3 \ ({\rm mod }\ 4) $ and let $n'$ be an odd positive integer such that $\gcd(q,n')=1 $ and   $ {\rm ord}_{n'}(q)$  is odd.     Let     $\beta$ be   the positive integer such that $2^\beta| |(q^2-1)$.   Let  $i$ be  a positive integer  and let 	$a$ be an odd integer.  Then one of  the following statements holds.

		\begin{enumerate}[$i)$] 
			
			\item  $Cl_{q,2^{i+1}n'}(a)$ and  $ Cl_{q,2^{i+1}n'}(a+2^in')$  induce distinct monic  irreducible polynomials of   degree $| Cl_{q,2^{i}n'}(a)|$ for all     $2\leq i\leq \beta-1.$
			\item  For each  $  i =1$ or  $i\geq \beta$, if   $f(x) $ is  induced by $Cl_{q,2^{i}n'}(a)$, then   $Cl_{q, 2^{i+1}n'} (a)   $     induces  $f(x^2)$.
		\end{enumerate}    
\end{lemma}
\begin{proof}
	To prove $i)$, assume that  $2\leq i\leq \beta-1$.   By Proposition \ref{propQ3} $i.a)$,   we have $Cl_{q,2^{i+1}n'}(a)\neq Cl_{q,2^{i+1}n'}(a+2^in')$.  From Lemma \ref{lemCardinal}, it follows that $| Cl_{q,2^{i+1}n'}(a)|=| Cl_{q,2^{i+1}n'}(a+2^in')|$ which equals to   $| Cl_{q,2^{i}n'}(a)|$ by  the proof of Proposition \ref{propQ3} $i.a)$. Hence,   $Cl_{q,2^{i+1}n'}(a)$ and  $ Cl_{q,2^{i+1}n'}(a+2^in')$  induce distinct monic  irreducible polynomials of   degree $| Cl_{q,2^{i}n'}(a)|$. 
	
	To proof $ii)$, assume that   $  i =1$ or $i\geq \beta$.  Assume that $f(x) $ is  induced by $Cl_{q,2^{i}n'}(a)$. Let $\alpha$ be a $2^{i+1} n'$th root of unity. Then $\alpha^2$ is a $2^{i} n'$th root of unity  and $f(x)=\prod\limits_{j\in Cl_{q,2^{i}n'}(a)} (x-(\alpha^2)^j)$.  From Proposition \ref{propQ3} $i.b)$,  we have  $Cl_{q, 2^{i+1}n'} (a)     = Cl_{q, 2^{i}n'} (a) \cup(Cl_{q, 2^{i}n'} (a) +2^{i}n')$.  It follows that 
	\begin{align*}
	\prod_{j\in Cl_{q,2^{i+1}n'}(a)} (x-\alpha^j)&= 	\prod_{j\in Cl_{q, 2^{i}n'} (a) \cup\{ Cl_{q, 2^{i}n'} (a) +2^{i}n'\}} (x-\alpha^j)\\
	&= 	\prod_{j\in Cl_{q, 2^{i}n'} (a)  } (x-\alpha^j) \times 	\prod_ {j\in \{ Cl_{q, 2^{i}n'} (a) +2^{i}n'\} }(x-\alpha^j)\\
	&= 	\prod_{j\in Cl_{q, 2^{i}n'} (a)  } (x-\alpha^j) (x-\alpha^{j+2^in'})\\
		&= 	\prod_{j\in Cl_{q, 2^{i}n'} (a)  } (x-\alpha^j) (x+\alpha^{j})\\
			&= 	\prod_{j\in Cl_{q, 2^{i}n'} (a)  } (x-\alpha^{2j}) \\
			&=f(x^2).
	\end{align*}
	Therefore,  $Cl_{q, 2^{i+1}n'} (a)   $     induces  $f(x^2)$  as desired. 
\end{proof}
The next corollary can be deduced directly from the above lemma.
\begin{corollary}  Assume the notations as in Lemma \ref{lemClPoly3}  with  $i\geq \beta$.  
If   $f(x) $ is  induced by $Cl_{q,2^{i}n'}(a)$, then    $f(x^{2^j})$  is irreducible  for all   $j\geq \beta -i$. 
\end{corollary}

In order to simplify the notations in the next theorem,  let $ \alpha$ and  $\gamma$  be $2^in'$th and  $2^{i+1}n'$th roots of unity, respectively.   For each  ${a\in SO_{q}(2^in')} $,  let  

\begin{align} \label{fg} f_a(x)=\prod_{j\in Cl_{q, 2^{i}n'} (a)  } (x-\alpha^j) \text{ and } g_j(x)=\prod_{j\in Cl_{q, 2^{i+1}n'} (a)  } (x-\gamma^j) \end{align} be the irreducible polynomials induced by  $Cl_{q, 2^{i}n'} (a) $ and $Cl_{q, 2^{i+1}n'} (a) $, respectively.  Using  these notations,  a recursive factorization of  $x^{2^in'}+1$ is given as follows.

\begin{theorem}  \label{thm-main-Fact3} Let $q$ be a prime  power such that $q \equiv 3 \ ({\rm mod }\ 4) $ and let $n'$ be an odd positive integer such that $\gcd(q,n')=1 $ and   $ {\rm ord}_{n'}(q)$  is odd.    Let     $\beta$ be   the positive integer such that $2^\beta| |(q^2-1)$.  Then   the following statements hold.

	\begin{enumerate}[$i)$]
		\item  If $i=0$,  then \begin{align} x^{2^i n'}+1  = x^{n'}+1   = 
		\prod\limits_{a\in SO_{q}(2n')} f_a(x)\end{align}

		\item If  $i\geq 1$, then 
		\begin{align}
		x^{2^in'}+1=   \begin{cases}
		\prod\limits_{a\in SO_{q}(2^in')} f_{a}(x^2) & \text{ if } i=1 \text{ or } i\geq \beta,\\
		\prod\limits_{a\in SO_{q}(2^in')}g_{a}(x) g_{a+2^in'}(x) & \text{ if } 2\leq i\leq \beta-1,
		\end{cases}
		\end{align} 
		where   $f_a(x)$ and $g_a(x)$ are given  in \eqref{fg}.
	\end{enumerate}	   In this case,  we have 
\[	x^{2^{\beta-1+i}n'}+1=   
\prod\limits_{a\in SO_{q}(2^{\beta}n')}f_{a}(x^{2^i})  \] 
for all $i\geq 0$.
\end{theorem}
\begin{proof}
From \eqref{facn0}, we note that  
	 \begin{align} x^{2^i n'}+1  =  
	\prod\limits_{a\in SO_{q}(2^{i+1}n')} f_a(x).\end{align} 
	The first statement is  the special case  where $i=0$.
	From Proposition \ref{propQ3} $i)$, it can be deduced that     \begin{align*}
	SO_{q}(2^{i+1}n') =\begin{cases}
	 	SO_{q}(2^{i}n')   &\text{ if }  i=1 \text{ or } i\geq \beta,	\\
	 	SO_{q}(2^{i}n') \cup 	(SO_{q}(2^{i}n')+2^in')  &\text{ if }  2 \leq i\leq \beta-1,
		\end{cases}
		\end{align*}
		where the union is disjoint.   The  results therefore follow from Lemma \ref{lemClPoly3}.
\end{proof}

A recursive  formula for   the  number of monic irreducible factors of $x^{2^in'}+1$ over $\mathbb{F}_q$ follows immediately from the theorem.  

\begin{corollary} \label{corEnumQ3} Let $q$ be a prime  power  such that $q \equiv 3 \ ({\rm mod }\ 4) $ and let $n'$ be an odd positive integer such that $\gcd(q,n')=1 $ and   $ {\rm ord}_{n'}(q)$  is odd.    Let    $i\geq 0$ be an integer and let  $\beta$ be   the positive integer such that $2^\beta| |(q^2-1)$.  Then  
	\begin{align} \label{Nq-1}
	N_q(n')=\sum_{d\mid n'} \frac{\phi(2d)}{{\rm  ord}_{2d}(q)}
	\end{align}	
	and  
	 
		\begin{align} \label{Nq-3}
		N_q(2^in')=\begin{cases}
		N_q(n')   & \text{ if } i=1,\\
		2N_q(2^{i-1}n') = 2^{i-1} N_q(n') & \text{ if } 2\leq i\leq \beta-1,\\
		N_q(2^{\beta-2}n') =2^{\beta-2}N_q(n') & \text{ if } i\geq \beta.
		\end{cases}
		\end{align}
 
\end{corollary}
\begin{proof}
	Equation \eqref{Nq-1} is a special case of  \eqref{nqn'}.   Equation  \eqref{Nq-3} follows immediately from Theorem \ref{thm-main-Fact3}. 
\end{proof}

\subsubsection{$q \equiv 1 \ ({\rm mod }\ 4) $}
 Here, we focus on  $q \equiv 1 \ ({\rm mod }\ 4) $. First,  some  useful  relations between the $q$-cyclotomic coset $Cl_{q,2^{i+1}n'}(a)$  and its   induced  polynomial are established.

\begin{lemma}   \label{lemClPoly1}  Let $q$ be a prime  power such that $q \equiv 1 \ ({\rm mod }\ 4) $  and let $n'$ be an odd positive integer such that $\gcd(q,n')=1 $ and   $ {\rm ord}_{n'}(q)$  is odd.     Let     $\beta$ be   the positive integer such that $2^\beta| |(q^2-1)$.   Let  $i$ be  a positive integer  and let 	$a$ be an odd integer.  Then one of  the following statements holds.

	\begin{enumerate}[$i)$] 
			\item  $Cl_{q,2^{i+1}n'}(a)$ and  $ Cl_{q,2^{i+1}n'}(a+2^in')$  induce distinct monic  irreducible polynomials of the same degree  for all     $1\leq i\leq \beta-2.$
		\item For each    $  i \geq \beta-1$, if $f(x) $ is   induced by $Cl_{q,2^{i}n'}(a)$, then $Cl_{q, 2^{i+1}n'} (a)   $     induce  $f(x^2)$.
	 \end{enumerate}
\end{lemma}
\begin{proof}
	The proof can be obtained using the arguments similar to those in the proof of Lemma \ref{lemClPoly3} while Proposition    \ref{propQ1} $i)$ is applied instead of    Proposition~\ref{propQ3}~$i)$.
\end{proof}
\begin{corollary}  
	Assume the notations as in Lemma \ref{lemClPoly1}  with  $i\geq \beta-1$.  
	If   $f(x) $ is  induced by $Cl_{q,2^{i}n'}(a)$, then    $f(x^{2^j})$  is irreducible  for all   $j\geq \beta -i-1$. 
\end{corollary}
 
The factorization  of  $x^{2^in'}+1$ is given in the next theorem. 
\begin{theorem}  \label{thm-main-Fact1} Let $q$ be a prime  power such that $q \equiv 1 \ ({\rm mod }\ 4) $ and let $n'$ be an odd positive integer such that $\gcd(q,n')=1 $ and   $ {\rm ord}_{n'}(q)$  is odd.    Let     $\beta$ be   the positive integer such that $2^\beta| |(q^2-1)$.  Then   the following statements hold.

	\begin{enumerate}[$i)$]
		\item  If $i=0$,  then \begin{align} x^{2^i n'}+1  = x^{n'}+1   = 
		\prod\limits_{a\in SO_{q}(2n')} f_a(x)\end{align}

		\item If  $i\geq 1$, then 
		\begin{align}
		x^{2^in'}+1=   \begin{cases}
		\prod\limits_{a\in SO_{q}(2^in')}g_{a}(x) g_{a+2^in'}(x) & \text{ if } 1\leq i\leq \beta-2,\\
		\prod\limits_{a\in SO_{q}(2^in')} f_{a}(x^2) & \text{ if }  i\geq \beta-1,
		\end{cases}
		\end{align} 
		where   $f_a(x)$ and $g_a(x)$ are given in \eqref{fg}.
	\end{enumerate}	    In this case,  we have 
	\[	x^{2^{\beta-2+i}n'}+1=   
	\prod\limits_{a \in SO_{q}(2^{\beta-1}n')}f_{a}(x^{2^i})  \] 
	for all $i\geq 0$.
\end{theorem}

\begin{proof}
		The proof can be obtained using the arguments similar to those in the proof of Theorem \ref{thm-main-Fact3} while Proposition    \ref{propQ1} $i)$ and Lemma \ref{lemClPoly1}    are applied instead of    Proposition \ref{propQ3} $i)$ and Lemma \ref{lemClPoly3} .
\end{proof}

 From the theorem, the enumeration of monic irreducible factor of  $x^{2^in'}-1$ over $\mathbb{F}_q$ can be concluded in the next corollary.

\begin{corollary} \label{corEnumQ1} Let $q$ be a prime  power such that $q \equiv 1 \ ({\rm mod }\ 4) $   and let $n'$ be an odd positive integer such that $\gcd(q,n')=1 $ and   $ {\rm ord}_{n'}(q)$  is odd.    Let    $i\geq 0$ be an integer and let  $\beta$ be   the positive integer such that $2^\beta| |(q^2-1)$.  Then  
	\begin{align} \label{Nq-11}
	N_q(n')=\sum_{d\mid n'} \frac{\phi(2d)}{{\rm  ord}_{2d}(q)}
	\end{align}	
	and 
		
			\begin{align} \label{Nq-2}
		N_q(2^in')=\begin{cases}
		2N_q(2^{i-1}n') = 2^{i} N_q(n') & \text{ if } 1\leq i\leq \beta-2,\\
		N_q(2^{\beta-2}n') =2^{\beta-2}N_q(n') & \text{ if } i\geq \beta-1.
		\end{cases}
		\end{align}

\end{corollary}
\begin{proof}
	Equation \eqref{Nq-11} is given in  \eqref{nqn'}.    Equation  \ref{Nq-2}   follow immediately from Theorem \ref{thm-main-Fact1}. 
\end{proof}

 \subsection{Factorization of $x^n+1$  over $\mathbb{F}_q$ with  Even  ${\rm ord}_{n'}(q)$} \label{subsec4.1}

In this subsection,  we focus on the case where  ${\rm ord}_{n'}(q)$ is even, i.e.,  $2^\lambda ||{\rm ord}_{n'}(q)$ for some positive integer $\lambda$.   
The results are not strong as     the previous subsection.  Precisely,  a  recursive factorization of $x^{2^in'}+1$  over $\mathbb{F}_q$  is given only  for all sufficiently large positive integers $i$.

In general,  the factorization of $x^{2^in'}+1$  over $\mathbb{F}_q$  is given  in \eqref{facn0}. For $i\geq \lambda+\beta-1$,  a simpler recursive method for the factorization is given in the next theorem.  

\begin{theorem} \label{thm-evenO}Let $q$ be an odd  prime  power and let $n'$ be an odd positive integer such that $\gcd(q,n')=1 $.    Let  $\lambda$ be the positive integer such that   $ 2^\lambda|| {\rm ord}_{n'}(q)$  and let   $\beta$ be   the positive integer such that $2^\beta| |(q^2-1)$. 
   Then 		
			\[	x^{2^{\lambda+\beta-1+j}n'}+1=   
		\prod\limits_{a\in SO_{q}(2^{\lambda+\beta}n')}f_{a}(x^{2^j})  \] 
		for all $j\geq 0$.
\end{theorem}
\begin{proof}
			The proof can be obtained using the arguments similar to those in the proof of Theorem \ref{thm-main-Fact3} while Proposition    \ref{propQ1} $ii)$ and  Proposition \ref{propQ3} $ii)$      are applied instead of    Proposition    \ref{propQ1} $i)$ and  Proposition \ref{propQ3} $i)$.
\end{proof}

The next corollary follows immediately. 
\begin{corollary}  \label{corEnumQAll}
Let $q$ be an odd  prime  power and let $n'$ be an odd positive integer such that $\gcd(q,n')=1 $.    Let  $\lambda$ be the positive integer such that   $ 2^\lambda|| {\rm ord}_{n'}(q)$  and let   $\beta$ be   the positive integer such that $2^\beta| |(q^2-1)$.   Then  
	\begin{align*}
	N_q(2^i n')=  	N_q(2^{\lambda+ \beta-2}n')  	\end{align*}
	for all $ i\geq \lambda +\beta-1.$
\end{corollary}

\subsection{Algorithm and Examples}
In this subsection,  the above results are summarized  as an  algorithm for   factorizing     $x^{2^in'}+1$ over $\mathbb{F}_q$.   Some illustrative examples are given as well.

An algorithm  for  the factorization     of  $x^{2^in'}+1$ over $\mathbb{F}_q$   is  given in Algorithm~\ref{AL1}.

\renewcommand{\figurename}{Algorithm}
\begin{figure}[!hbt]
		\caption{Algorithm for the Factorization of  $x^{2^in'}+1$ over $\mathbb{F}_q$}\label{AL1}
		\medskip
	\hrule
	
 	\medskip  
 
	{\scriptsize  \tt  
	Input: odd prime power $q$, odd integer  $n'$ with $\gcd(q,n')=1$,  and    integer $i\geq 0$.
	\begin{enumerate} [1)]
			\item  Compute the positive integer  $\beta$ such that   $ 2^\beta || (q^2-1)$.
		\item  Compute ${\rm ord}_{n'}(q)$ and  the  integer  $\lambda$ such that   $ 2^\lambda|| {\rm ord}_{n'}(q)$.
		\item Consider the following cases:
		\begin{enumerate}[I)]
			\item   $\lambda=0$. 
			\begin{enumerate}[i)]
				\item  $q \equiv 1 \ ({\rm mod }\ 4) $.  
				\begin{enumerate}[a)]
					\item $i=0$.  Compute
					$ x^{2^i n'}+1  = x^{n'}+1   = 
					\prod\limits_{a\in SO_{q}(2n')} f_a(x).$
				 		
					\item $1\leq  i\leq\beta-2 $.    Compute 
						\begin{align*}
					x^{2^in'}+1=   \prod\limits_{a\in SO_{q}(2^in')}g_{a}(x) g_{a+2^in'}(x)
					\end{align*} 
					and $SO_{q}(2^{i+1}n') =SO_{q}(2^{i}n') \cup 	(SO_{q}(2^{i}n')+2^in') $.

					\item $i\geq \beta-1$. Compute  	
						\begin{align*}
					x^{2^in'}+1=   \prod\limits_{a\in SO_{q}(2^{\beta-1}n')} f_{a}(x^{2^{i-\beta+2}}) .
					\end{align*} 			
				\end{enumerate} 
	 
				\item $q \equiv 3 \ ({\rm mod }\ 4) $.   
						\begin{enumerate}[a)]
							\item $ 0\leq i\leq 1$.  Compute 
							$ x^{2^i n'}+1  =
							\prod\limits_{a\in SO_{q}(2n')} f_a(x^i).$ 
														 
										\item $2\leq i\leq \beta-1$.  Compute
									\begin{align*}
							x^{2^in'}+1=   \prod\limits_{a\in SO_{q}(2^in')}g_{a}(x) g_{a+2^in'}(x)
							\end{align*} 
							and $SO_{q}(2^{i+1}n') =SO_{q}(2^{i}n') \cup 	(SO_{q}(2^{i}n')+2^in') $.
		
					\item $i\geq \beta$.   Compute 
						\begin{align*}
					x^{2^in'}+1=   \prod\limits_{a\in SO_{q}(2^\beta n')} f_{a}(x^{2^{i-\beta+1}}) .
					\end{align*} 		
				\end{enumerate}
			\end{enumerate}
			
			\item   $\lambda\geq 1$.
			
				\begin{enumerate}[i)]
				\item  $0\leq i\leq \lambda+\beta-2$.  Compute  $x^{2^i n'}+1$  directly using \eqref{facn0}
				\item  $i\geq \lambda+\beta-1$.   Compute
					\begin{align*}
				x^{2^in'}+1=   \prod\limits_{a\in SO_{q}(2^{\lambda+\beta}n')} f_{a}(x^{2^{i-\lambda-\beta+1}}) .
				\end{align*} 					
			\end{enumerate}
		\end{enumerate}
	\end{enumerate}
Note that where   $f_a(x)$ and $g_a(x)$ are given in \eqref{fg}.
}
\medskip
	\hrule
\end{figure}

For the enumeration of  monic  irreducible factors of $x^{2^in'}+1$ over   $\mathbb{F}_q$, it can be calculated using \eqref{nqn'}. With more information on $n'$, $i$, and $q$,  the formula can be simplified using Corollaries  \ref{corEnumQ3},  \ref{corEnumQ1}, and \ref{corEnumQAll} of the form

\begin{align}\label{eqsum}
N_q(2^in')=\begin{cases}
	  2^{i} N_q(n') & \text{ if } \lambda=0,~ 1\leq i\leq \beta-2 \text{ and } q \equiv 1 \ ({\rm mod }\ 4) ,\\
 2^{\beta-2}N_q(n') & \text{ if } \lambda=0,~  i\geq \beta-1 \text{ and } q \equiv 1 \ ({\rm mod }\ 4) \\
N_q(n')   & \text{ if } \lambda=0,~  i=1 \text{ and } q \equiv 3 \ ({\rm mod }\ 4) ,\\
  2^{i-1} N_q(n') & \text{ if } \lambda=0,~  2\leq i\leq \beta-1 \text{ and } q \equiv 3 \ ({\rm mod }\ 4) ,\\
 2^{\beta-2}N_q(n') & \text{ if }\lambda=0,~   i\geq \beta \text{ and } q \equiv 3 \ ({\rm mod }\ 4) ,\\
	N_q(2^{\lambda+ \beta-2}n')   & \text{ if }  \lambda\geq 1 \text{ and }   i\geq \lambda +\beta-1,
\end{cases}
\end{align}
where   $\lambda$ is  the positive integer such that   $ 2^\lambda|| {\rm ord}_{n'}(q)$,   $\beta$ is    the positive integer such that $2^\beta| |(q^2-1)$, and 
	\begin{align*} 
N_q(n')=\sum_{d\mid n'} \frac{\phi(2d)}{{\rm  ord}_{2d}(q)}.
\end{align*}	

From \eqref{eqsum},   the number $	N_q(2^in')$ of monic  irreducible factors of $x^{2^in'}+1$ over   $\mathbb{F}_q$  becomes a constant independent of $i$ for all $i\geq \lambda +\beta-1$ if  $\lambda=0$ and $q \equiv 3 \ ({\rm mod }\ 4)$, and  for all $i\geq \lambda +\beta-2$ otherwise.  Illustrative examples for the number $	N_q(2^in')$ of monic  irreducible factors of $x^{2^in'}+1$ over   $\mathbb{F}_q$  with  odd $ {\rm ord}_{n'}(q)$ and even  $ {\rm ord}_{n'}(q)$ are given in Table \ref{T1} and Table   \ref{T2}, respectively. 

\begin{table}[!hbt]
		\caption{$	N_q(2^in')$ of monic  irreducible factors of $x^{2^in'}+1$ over   $\mathbb{F}_q$  with odd $ {\rm ord}_{n'}(q)$} \label{T1}
		\centering \scriptsize
	\begin{tabular}{|c|c|c|c|c|c|c|}
		\hline
		$q$& $n'$ &$ {\rm ord}_{n'}(q)$& $\lambda $ &$\beta$  & $i$ & $	N_q(2^in')$

 \\ \hline $ 3 $ & $ 1 $ & $ 1 $ & $ 0 $ & $ 3 $ & $ 0 $ & $ 1 $  
 \\  &&&&& $ 1 $ & $ 1 $  
 \\  &&&&& $\geq  2 $ & $ 2 $

 \\ \hline $ 3 $ & $ 11 $ & $ 5 $ & $ 0 $ & $ 3 $ & $ 0 $ & $ 3 $ 
 \\  &&&&& $ 1 $ & $ 3 $  
 \\  &&&&& $ \geq 2 $ & $ 6 $

 \\ \hline $ 3 $ & $ 13 $ & $ 3 $ & $ 0 $ & $ 3 $ & $ 0 $ & $ 5 $ 
 \\  &&&&& $ 1 $ & $ 5 $  
 \\  &&&&& $ \geq 2 $ & $ 10 $

 \\ \hline $ 5 $ & $ 1 $ & $ 1 $ & $ 0 $ & $ 3 $ & $ 0 $ & $ 1 $ 
 \\  &&&&& $ \geq 1 $ & $ 2 $

 \\ \hline $ 5 $ & $ 11 $ & $ 5 $ & $ 0 $ & $ 3 $ & $ 0 $ & $ 3 $ 
 \\  &&&&& $ \geq 1 $ & $ 6 $

 \\ \hline $ 7 $ & $ 1 $ & $ 1 $ & $ 0 $ & $ 4 $ & $ 0 $ & $ 1 $ 
 \\  &&&&& $ 1 $ & $ 1 $  
 \\  &&&&& $ 2 $ & $ 2 $  
 \\  &&&&& $ \geq 3 $ & $ 4 $

 \\ \hline $ 7 $ & $ 3 $ & $ 1 $ & $ 0 $ & $ 4 $ & $ 0 $ & $ 3 $ 
 \\  &&&&& $ 1 $ & $ 3 $  
 \\  &&&&& $ 2 $ & $ 6 $  
 \\  &&&&& $ \geq 3 $ & $ 12 $

 \\ \hline $ 7 $ & $ 9 $ & $ 3 $ & $ 0 $ & $ 4 $ & $ 0 $ & $ 5 $ 
 \\  &&&&& $ 1 $ & $ 5 $  
 \\  &&&&& $ 2 $ & $ 10 $  
 \\  &&&&& $ \geq 3 $ & $ 20 $ 
 
 \\ \hline $ 9 $ & $ 1 $ & $ 1 $ & $ 0 $ & $ 4 $ & $ 0 $ & $ 1 $ 
 \\  &&&&& $ 1 $ & $ 2 $  
 \\  &&&&& $ \geq 2 $ & $ 4 $  
 
 \\ \hline $ 9 $ & $ 7 $ & $ 3 $ & $ 0 $ & $ 4 $ & $ 0 $ & $ 3 $ 
 \\  &&&&& $ 1 $ & $ 6 $  
 \\  &&&&& $ \geq 2 $ & $ 12 $

 \\ \hline $ 9 $ & $ 11 $ & $ 5 $ & $ 0 $ & $ 4 $ & $ 0 $ & $ 3 $ 
 \\  &&&&& $ 1 $ & $ 6 $  
 \\  &&&&& $\geq  2 $ & $ 12 $

 \\ \hline $ 9 $ & $ 13 $ & $ 3 $ & $ 0 $ & $ 4 $ & $ 0 $ & $ 5 $ 
 \\  &&&&& $ 1 $ & $ 10 $  
 \\  &&&&& $ \geq 2 $ & $ 20 $  
 
	\\	\hline
	\end{tabular}

\end{table}	
In Table \ref{T1}, the results for the  $q\in \{3,7\}$ and $q\in \{5,9\}$ are obtained  from Corollary \ref{corEnumQ3} and 
Corollary \ref{corEnumQ1}, respectively.

 \begin{table}[!hbt]
 	\caption{$	N_q(2^in')$ of  monic irreducible factors of $x^{2^in'}+1$ over   $\mathbb{F}_q$ with  even $ {\rm ord}_{n'}(q)$} \label{T2}
 	\centering \scriptsize
 	\begin{tabular}{|c|c|c|c|c|c|c|}
 		\hline
 		$q$& $n'$ &$ {\rm ord}_{n'}(q)$& $\lambda $ &$\beta$  & $i$ & $	N_q(2^in')$  
 		
 \\ \hline $ 3 $ & $ 5 $ & $ 4 $ & $ 2 $ & $ 3 $ & $ 0 $ & $ 2 $ 
 \\  &&&&& $ 1 $ & $ 3 $ 
 \\  &&&&& $ 2 $ & $ 6 $ 
 \\  &&&&& $ \geq 3 $ & $ 10 $

 \\ \hline $ 3 $ & $ 7 $ & $ 6 $ & $ 1 $ & $ 3 $ & $ 0 $ & $ 2 $ 
 \\  &&&&& $ 1 $ & $ 3 $ 
 \\  &&&&& $\geq  2 $ & $ 6 $

 \\ \hline $ 5 $ & $ 3 $ & $ 2 $ & $ 1 $ & $ 3 $ & $ 0 $ & $ 2 $ 
 \\  &&&&& $ 1 $ & $ 4 $ 
 \\  &&&&& $\geq  2 $ & $ 6 $

 \\ \hline $ 5 $ & $ 7 $ & $ 6 $ & $ 1 $ & $ 3 $ & $ 0 $ & $ 2 $ 
 \\  &&&&& $ 1 $ & $ 4 $ 
 \\  &&&&& $ \geq 2 $ & $ 6 $

 \\ \hline $ 5 $ & $ 9 $ & $ 6 $ & $ 1 $ & $ 3 $ & $ 0 $ & $ 3 $ 
 \\  &&&&& $ 1 $ & $ 6 $ 
 \\  &&&&& $\geq  2 $ & $ 10 $

 \\ \hline $ 5 $ & $ 13 $ & $ 4 $ & $ 2 $ & $ 3 $ & $ 0 $ & $ 4 $ 
 \\  &&&&& $ 1 $ & $ 8 $ 
 \\  &&&&& $ 2 $ & $ 14 $ 
 \\  &&&&& $\geq  3 $ & $ 26 $

 \\ \hline $ 7 $ & $ 5 $ & $ 4 $ & $ 2 $ & $ 4 $ & $ 0 $ & $ 2 $ 
 \\  &&&&& $ 1 $ & $ 3 $ 
 \\  &&&&& $ 2 $ & $ 6 $ 
 \\  &&&&& $ 3 $ & $ 12 $ 
 \\  &&&&& $ \geq 4 $ & $ 20 $

 \\ \hline $ 7 $ & $ 11 $ & $ 10 $ & $ 1 $ & $ 4 $ & $ 0 $ & $ 2 $ 
 \\  &&&&& $ 1 $ & $ 3 $ 
 \\  &&&&& $ 2 $ & $ 6 $ 
 \\  &&&&& $ \geq 3 $ & $ 12 $

 \\ \hline $ 7 $ & $ 13 $ & $ 12 $ & $ 2 $ & $ 4 $ & $ 0 $ & $ 2 $ 
 \\  &&&&& $ 1 $ & $ 3 $ 
 \\  &&&&& $ 2 $ & $ 6 $ 
 \\  &&&&& $ 3 $ & $ 12 $ 
 \\  &&&&& $ \geq 4 $ & $ 20 $

 \\ \hline $ 7 $ & $ 15 $ & $ 4 $ & $ 2 $ & $ 4 $ & $ 0 $ & $ 6 $ 
 \\  &&&&& $ 1 $ & $ 9 $ 
 \\  &&&&& $ 2 $ & $ 18 $ 
 \\  &&&&& $ 3 $ & $ 36 $ 
 \\  &&&&& $\geq  4 $ & $ 60 $  
 
 \\ \hline $ 9 $ & $ 5 $ & $ 2 $ & $ 1 $ & $ 4 $ & $ 0 $ & $ 3 $ 
 \\  &&&&& $ 1 $ & $ 6 $ 
 \\  &&&&& $ 2 $ & $ 12 $ 
 \\  &&&&& $ \geq 3 $ & $ 20 $

 		\\	\hline
 	\end{tabular}
 	
 \end{table}	
 In Table \ref{T2},  the last row of each $n'$ is obtained from Corollary  \ref{corEnumQAll}. Otherwise, it is computed using \eqref{nqn'}.

\section{Applications}  \label{sec5}

In this section,  the factorization of $x^{2^in'}+1$ over $\mathbb{F}_q$ obtained in Section \ref{sec4}  are applied in the study of negacyclic codes. Some known results are revisited in simpler  forms. 


A linear code of length $n$ over $\mathbb{F}_q$ is defined to be a subspace of the the $\mathbb{F}_q$-vector space $\mathbb{F}_q^n$.   The {\em dual} of a linear code $C$ of length $n$ over $\mathbb{F}_q$ is defined to be \[C^\perp=\{ (v_0,v_1,\dots,v_{n-1} ) \in \mathbb{F}_q^n \mid \sum\limits_{i=0}^{n-1} c_iv_i =0 \text{ for all }(c_0,c_1,\dots,c_{n-1})\in C\}.\]  A linear code $C$ is said to be {\em self-dual} if $C=C^\perp$ and it is said to be {complementary dual} if $C\cap C^\perp=\{0\}$.

A linear code $C$ of length $n$ over $\mathbb{F}_q$ is said to be {\em negacyclic} if it is closed under the negacyclic shift. Precisely, $(-c_{n-1},c_0,c_1,\dots,c_{n-2})\in C$ for every  $(c_0,c_1,\dots,c_{n-2},c_{n-1})\in C$.  Under the map $\pi: \mathbb{F}_q^n \to  \mathbb{F}_q[x]/\langle x^n+1\rangle $ defined by 
\[ (c_0,c_1,\dots,c_{n-2},c_{n-1})\mapsto c_0+c_1x+c_2x^2+\dots +c_{n-1} x^{n-1} ,  \]
it is well known (see \cite{SJLU2015})  that a linear code $C$ of length $n$ over $\mathbb{F}_q$ is negacyclic if and only if $\pi(C)$ is an ideal in the principal ideal ring  $\mathbb{F}_q[x]/\langle x^n+1\rangle$.  The map $\pi$  induces  a one-to-one correspondence between negacyclic codes of length $n$ over $\mathbb{F}_q$ and ideas in $\mathbb{F}_q[x]/\langle x^n+1\rangle$.  In this case, $\pi(C)$ is uniquely  generated by the monic divisor of $x^n+1$ of minimal degree in  $\pi(C)$.  The such polynomial is call the {\em generator polynomial} of $C$.

Let $q$ be an odd  prime  power and let $n'$ be an odd positive integer such that $\gcd(q,n')=1 $.    Let  $\lambda$ be the positive integer such that   $ 2^\lambda|| {\rm ord}_{n'}(q)$  and let   $\beta$ be   the positive integer such that $2^\beta| |(q^2-1)$.   Let     \[k=\begin{cases}
\lambda+\beta -1& \text{ if  } \lambda=0 \text{ and }  q\equiv 3 \ ({\rm mod } 4,\\
\lambda+\beta-2 & \text{ otherwise}.
\end{cases}\]
In general, 
negacyclic codes have been studied in 
\cite{BJ2019}, \cite{JPR2019}, and \cite{SJLU2015}.  Here, we focus on negacyclic codes of length $n=p^s2^in'$ with $i\geq k$, where $p$ is the characteristic of $\mathbb{F}_q$. The construction and enumeration of such negacyclic codes are simplified using the results from Section \ref{sec4}.

From   \eqref{facn0}, we  have \begin{align}  
x^{2^k n^\prime}+1 = 
\prod\limits_{j=1} ^{N_q(2^kn')} r_j(x).
\end{align}
 Based on Theorem  \ref{thm-main-Fact3}, Theorem \ref{thm-main-Fact1},  and   Theorem  \ref{thm-evenO},   it follows that
 \begin{align} 
x^{p^s2^{i} n'}+1=(x^{2^{i} n^\prime}+1)^{p^s} = 
\prod\limits_{j=1} ^{N_q(2^kn')} (r_j(x^{2^{i-k}}))^{p^s}
\end{align}
 and $r_j(x^{2^{i-k}})$  is irreducible  for all $i\geq k$.

The following characterization and enumeration of     negacyclic codes of length $n=p^s2^in'$ with $i\geq k$ are straightforward. The proof is committed.  
 \begin{theorem}   Assume the notations above. The the following statements hold.
 	
 	\begin{enumerate}
 		\item   The map $T:   \mathbb{F}_q[x]/\langle x^{p^s2^k n'}+1\rangle \to  \mathbb{F}_q[x]/\langle x^{p^s2^i n'}+1\rangle$ defined by \[f(x) \mapsto f(x^{2^{i-k}})\] is a ring isomorphism for all integers  $i\geq k$. 
 		
 		\item For each integer $i\geq k$,  $g(x)$ is the   generator polynomial of a    negacyclic code of length $p^s2^k n'$ over $\mathbb{F}_q$  if and only if  $g(x^{2^{i-k}})$ is the   generator polynomial of a    negacyclic code of length $p^s2^i n'$ over $\mathbb{F}_q$ 
 		\item   The number of  negacyclic codes of length $p^s2^in'$ over $\mathbb{F}_q$ is 
\[{(p^s+1)}^{N_q(2^kn')}\] for all  $i\geq k$. 
 	\end{enumerate}
\end{theorem}
From the theorem, all negacyclic codes of length $p^s2^in'$ over $\mathbb{F}_q$  with $i\geq k $ can be determined using the negacyclic  codes  of length $p^s2^kn'$  over $\mathbb{F}_q$.

\end{document}